\documentclass[preprint,3p]{elsarticle}

\usepackage{amssymb}

\usepackage{color}
\usepackage{amsthm}
\usepackage{amssymb}
\usepackage{amsmath}
\usepackage{multirow}
\usepackage{graphicx}
\usepackage{lineno,hyperref}
\usepackage{booktabs}
\usepackage{tensor}
\usepackage{xcolor}
\usepackage{subfigure}

\theoremstyle{definition}

\theoremstyle{plain}
\newtheorem{theorem}{Theorem}[section]
\newtheorem{lemma}{Lemma}[section]
\newtheorem{remark}{Remark}[section]
\newtheorem{corollary}{Corollary}
\newtheorem{example}{Example}[section]

\journal{Journal }

\begin{document}

\begin{frontmatter}


 \title{Stability analysis of a second-order difference scheme for
the time-fractional mixed sub-diffusion and diffusion-wave equation
  }

                 \author[label1]{Anatoly A.
			Alikhanov\corref{cor1}}
		\cortext[cor1]{Corresponding author}
		\ead{aaalikhanov@gmail.com}
		\address[label1]{North-Caucasus Center for Mathematical Research, North-Caucasus Federal University, Pushkin str. 1,  Stavropol, 355017,  Russia}
\author[label1]{Mohammad Shahbazi Asl}
		\ead{mshahbazia@yahoo.com}
\author[label2]{Chengming Huang}
\ead{chengming\_huang@hotmail.com}
\address[label2]{School of Mathematics and Statistics, Huazhong University of Science and Technology, Wuhan, 430074, China}
 


%

\begin{abstract}

This study investigates a class of initial-boundary value problems 
 pertaining to  the time-fractional mixed sub-diffusion and diffusion-wave equation 
(SDDWE).
To facilitate the development of a numerical method and analysis, the original problem is transformed into a 
new integro-differential model which includes the Caputo derivatives and the Riemann-Liouville fractional integrals 
with orders belonging to $(0,1)$.
By providing an a priori estimate of the solution, we have established the existence and uniqueness of a numerical solution for the problem. 
We propose a second-order method to approximate the fractional Riemann-Liouville integral
 and employ an L2 type formula
to approximate the Caputo derivative. 
This results in a  method with a temporal accuracy of second-order for approximating the considered model. 
The proof of the unconditional stability of the proposed difference scheme
is established.
Moreover, we demonstrate the proposed method's potential to construct and analyze a second-order L2-type numerical scheme for a broader class of the
time-fractional  mixed SDDWEs with 
 multi-term time-fractional derivatives.
Numerical results are presented to assess the accuracy of the method and validate the theoretical findings.

\end{abstract}




\end{frontmatter}




\section{Introduction} \label{sec:1}

The increasing utilization of time fractional partial differential equations (TFPDEs) in mathematical models has made them a crucial area of research in the field of fractional calculus. 
TFPDEs based models offer advantages in simulating memory, hereditary, non-local characteristics,  relaxation vibrations, spatial heterogeneity, anomalous diffusion, and long-range interaction in viscoelasticity, electrochemistry and ﬂuid mechanics
\cite{asl2019new,asl2021high}. 
These unique features make them not only more effective but also more accurate and realistic when employed in process modeling compared to classical integer-order equations
\cite{asl2017improved,asl2018novel}. This holds true particularly in diverse fields of science and technology, including medical  and biological systems, options pricing models in ﬁnancial markets, economics, signal processing control systems, and more
\cite{roohi2015switching,taheri2023finite}.

Diffusion, being the most common physical phenomenon in nature, is described using TFPDEs that formulate anomalous diffusion processes
\cite{yan2021identify}. 
Two significant subclasses of TFPDEs are the time-fractional sub-diffusion equations (TFSDEs) with a fractional order in the range of (0, 1) and the time-fractional diffusion-wave equations (TFDWEs) with a fractional order in the range of (1, 2). 
These subclasses have garnered significant attention from researchers due to their distinctive characteristics and their relevance in modeling various phenomena.
However, there are some fundamental real-life systems whose characteristics cannot be fully captured by
 employing   TFSDEs or TFDWEs independently. In such  cases, 
 the mixed time-fractional SDDWEs can be utilized as a effective mathematical model
 \cite{zhao2019anisotropic}.
 The two-term mixed time-fractional SDDWEs offer enhanced accuracy and adaptability compared to the single-term TFSDEs or TFDWEs
  \cite{du2021temporal,ding2021development}.
These models are particularly beneficial at representing anomalous diffusion processes and capturing the properties of media, including power-law frequency dependence
\cite{ezz2020numerical}.
 Additionally, they are adept at modeling various types of viscoelastic damping, modelling the unsteady flow of a fractional Maxwell fluid, and describing the behavior of an Oldroyd-B fluid
 \cite{shen2021two}.

In a bounded  convex polygonal domain  
$\Omega \subset \mathbf{R}^{d}$, $d = 1, 2, 3,$
 with the Lipschitz continuous boundary $\partial \Omega$ 
we investigate an initial-boundary value problem for the time-fractional mixed SDDWE, given by
\begin{align}
\label{eq1}
&
\partial_{0t}^{\alpha+1}u(\mbox{\boldmath${x}$}, t) + \varkappa\partial_{0t}^{\beta}u(\mbox{\boldmath${x}$}, t) + \mathcal{A}u(\mbox{\boldmath${x}$},t) = g(\mbox{\boldmath${x}$}, t),
\\&
u(\mbox{\boldmath${x}$}  ,t) = 0, 
\qquad \qquad \qquad \qquad \qquad \quad
\mbox{\boldmath${x}$}\in \partial \Omega,
\label{eq2}
\\&
u(\mbox{\boldmath${x}$}, 0) = \phi(\mbox{\boldmath${x}$}),\quad u_t(\mbox{\boldmath${x}$}, 0) = \psi(\mbox{\boldmath${x}$}), 
~ \quad
 \mbox{\boldmath${x}$}\in \overline \Omega,
\label{eq3}
\end{align}
where 
$
\varkappa\geq 0, $
and
$(\mbox{\boldmath${x}$}, t) \in \Omega\times(0,T].$
Meanwhile, the operators $\partial_{0t}^{\alpha+1}u(\mbox{\boldmath${x}$}, t)$ and $\partial_{0t}^{\beta}u(\mbox{\boldmath${x}$}, t)$ 
with 
$ 0<\alpha<1$ and $ 0< \beta<1$
denote the fractional order Caputo derivatives  defined by
\begin{align}
\partial_{0t}^{\nu}u(\mbox{\boldmath${x}$}, t) = 
\int\limits_{0}^{t}\omega_{n-\nu}^{(t-\xi)} \frac{\partial^n u(\mbox{\boldmath${x}$}, \xi)}{\partial \xi^n} d\xi, 
\quad
 n-1< \nu < n,
\end{align}
in which the kernel is denoted by
$\omega_{\nu}^{(t)}=t^{\nu -1}/\Gamma(\nu)$,
$t>0$.
The elliptic operator $\mathcal{A}$ in \eqref{eq1} 
is introduced as 
\begin{align}
 \mathcal{A}u(\boldsymbol{x},t) = -\text{div}{\left(p(\boldsymbol{x} \right) \, \text{grad}\, u(\boldsymbol{x},t))} + q(\boldsymbol{x})u(\boldsymbol{x},t)
 ,
\end{align}
with coefficients 
$ 0<c_1\leq p(\boldsymbol{x})\leq c_2, $
and
 $q(\boldsymbol{x}) \geq 0.$
In the Hilbert space $\mathcal{H} = L_2(\Omega)$ the operator $\mathcal{A}$ is self-adjoint and positive definite, that is
\begin{align*}
\mathcal{A} = \mathcal{A}^{*} \geq \varkappa_{\mathcal{A}} \mathcal{I},
\quad
\varkappa_{\mathcal{A}} > 0,
\end{align*}
where $\mathcal{I}$ is the identity operator from the Hilbert space $\mathcal{H}$.
The origin of equation \eqref{eq1} can be traced back to the
 classical telegraph equation, a second-order hyperbolic equation, with specific values of
$\alpha=2$ 
and 
$\beta=1$. The telegraph equation has been widely employed to model the vibrations of numerous structures, including machines, beams, and buildings
\cite{du2021temporal}.
 Additionally, equation  \eqref{eq1}  allows for the consideration of various types of FDEs, such as the fractional cable equation, the generalized Oldroyd-B fluid model in non-Newtonian fluids, the generalized Maxwell fluid model, and a heated generalized second-grade fluid model
 \cite{feng2019finite}.
Even though there have been some attempts to seek the analytical solution to the 
model \eqref{eq1}--\eqref{eq3}, the exact analytical answers that have been obtained remain exceedingly complicated and difficult to estimate.
Therefore, it becomes paramount to devise effective and precise numerical algorithms to investigate   such models 
\cite{feng2019finite}.

In the specific case where model \eqref{eq1}, reduces exclusively to a TFSDEs, 
these types of models find widespread application in describing anomalous diffusion processes and mechanics such as polymers, organisms, biopolymers, ecosystems, liquid crystals, proteins, and fractal and percolation clusters. 
\cite{ma2019efficient,hajimohammadi2021numerical}.
For the sufficiently smooth solutions, 
two high order time discretization schemes of order 
$3-\alpha$ and $4-\alpha$ are presented in 
\cite{gao2014new} and \cite{cao2015high} respectively, for approximation of the Caputo derivative.
 These schemes are subsequently applied to solve TFSDEs.
 However, neither of these papers provide error estimates for the proposed methods.
  Lv and Xu \cite{lv2016error} conducted a comprehensive error analysis for the numerical method presented in \cite{gao2014new}, offering valuable insights into its accuracy. 
 Yan et al. \cite{yan2018analysis} employed the L1 formula to estimate the Caputo fractional derivative and made certain modifications to the initial steps of the numerical methods utilized  for solving TFSDEs.
They provided a proof that for both smooth and nonsmooth  data, the modified numerical technique 
 meets  convergence order of $2-\alpha$.
 To address the challenge of solving TFSDEs, 
 Wang et al. \cite{wang2020two} presented two novel temporal discretization algorithms designed specifically for this type of equations.
These new schemes involved modifications to the initial steps and weights of the numerical methods proposed in \cite{gao2014new} and \cite{cao2015high}, enabling them to capture the singularities inherent in the solution of the problem. 
The error estimates for these newly proposed algorithms were established using the Laplace transform method. 
The authors of \cite{li2018finite} propose an implicit-explicit scheme that combines the fractional trapezoidal method for time discretization with a fast solver for spatial discretization, aiming to efficiently solve two-dimensional nonlinear TFSDEs. 
They conduct a thorough investigation into the stability and convergence properties of their  proposed scheme.
The authors of \cite{hendy2019novel} utilize the 
L2-1$_\sigma$ approximation formula to introduce a second-order finite difference method for solving nonlinear TFSDE with multiple delays.

In the specific case where model \eqref{eq1}, exclusively reduces to a TFDWEs, these models find widespread application in various domains, including fluid flow, continuous-time finance, unification of diffusion and wave propagation phenomena, fractional random walks, oil strata, acoustic and mechanical responses, electrodynamics, and more
\cite{huang2008time,agrawal2002solution}.
Two high-order compact finite difference schemes based on the Crank-Nicolson method to approximate a modified diffusion equation are developed in
\cite{alikhanov2021crank}.
The authors provide a priori estimates for the problem and demonstrate that the presented schemes achieve an order of accuracy of 2 and $2-\alpha$,
$(0<\alpha<1)$, in the temporal direction.
A discrete scheme was proposed by Sun and Wu to approximate TFDWEs in 
\cite{sun2006fully}.
They have employed the energy method to establish the stability and convergence rate of their proposed scheme.
In \cite{jin2016two}, the authors developed robust fully discrete schemes for the TFSDEs and TFDWEs.
They conducted a comprehensive error analysis and derived optimal error estimates for both smooth and nonsmooth initial data.
Zhang and Wang \cite{zhang2023numerical} employed the L1 formula to develop a compact finite difference scheme for solving the TFDWEs with time delay. They established the convergence and stability of the difference scheme and demonstrated that their method has an order of accuracy of $3-\alpha$,
$(1<\alpha<2)$.
Zhang et al. \cite{zhang2023local} applied a reduction method to convert the TFDWEs into a coupled system. They then constructed a fully discrete scheme for solving the TFDWEs by utilizing the L1 formula on graded meshe in the temporal direction and the compact difference scheme in the spatial direction.
In \cite{maurya2023high},
 an adaptive stable implicit difference scheme on a non-uniform mesh is constructed for solving TFDWEs. The scheme incorporates a new high-order approximation for the Caputo fractional derivative. The authors discuss the unconditional stability, and convergence of the proposed method.

Considering the exceptional modeling capability of the model \eqref{eq1}, there has been considerable interest within the research community to develop accurate numerical algorithms for its solution.
The authors of \cite{ezz2020numerical} employed shifted Legendre polynomials and   spectral collocation method to propose a numerical algorithm for solving the two-dimensional time-fractional mixed SDDWE with multi-term fractional derivatives.
In  \cite{feng2019finite}, the mixed L schemes and the finite element technique are employed to solve a 
the two-dimensional time-fractional mixed multi-term SDDWE that involves a time-space coupled derivative.
Sun et al. \cite{sun2020new},
 proposed an analytical technique  for investigating a finite difference method that relies  on the L1 formula to solve the 
 the time-fractional mixed SDDWE.
 The numerical approaches provided in \cite{feng2019finite} and  \cite{sun2020new}
  have an order of accuracy that is less than two and is dependent on the order of fractional derivatives.
In order to address the time-fractional mixed models, Ding initially developed two distinct second-order difference analogs to approximate the Caputo derivative for orders $0 < \alpha < 1$ and $1 < \beta < 2$.
 These difference analogs were then utilized to establish a temporal second-order method for solving the time-fractional mixed SDDWE.
  Using the energy method, the authors demonstrated that the stability and convergence of the difference scheme are guaranteed only for 
  $ \beta \in ( 1,\frac{-1+\sqrt{17} }{2}]$.
Du and Sun \cite{Du2021} employed the method of order reduction to transform the  time-fractional mixed multi-term SDDWE model into a new model that incorporates fractional integral and sub-diffusion terms.
 Subsequently, they employed the L2-1$_\sigma$ formula to solve the resulting  problem. 
 The authors demonstrated that their method exhibits second-order convergence in both the temporal and spatial directions.

To summarize, the key innovative aspects of this research can be stated as follows:
\begin{itemize}
  \item 
  The  time-fractional mixed SDDWE problem is transformed
into a new model which includes the fractional 
derivatives and integrals with orders
belonging to (0, 1).
  \item 
  An a priori estimate of the solution is presented.
  \item 
  A temporal second-order L2-type  difference method to approximate
the time-fractional mixed SDDWE is constructed.
  \item 
  The unconditional stability of the proposed difference scheme is investigated.
  \item 
  The  potential of this study to construct and analyze a temporal second-order
 numerical scheme for a more general form of the time-fractional mixed SDDWE
with multi-term time-fractional derivatives is demonstrated.
\end{itemize}

 \section{Methodology} \label{sec:2}

Before the methodology for solving model \eqref{eq1} is presented, 
the definition of the Riemann-Liouville fractional integration operator  is recalled, which is given by 
\begin{align} \label{rla}
D_{0t}^{-\alpha}u(\mbox{\boldmath${x}$}, t)=\int\limits_{0}^{t}\omega_{\alpha}^{(t-\xi)} u(\mbox{\boldmath${x}$}, \xi) d\xi
.
\end{align}
Applying  
 the Riemann-Liouville integral of order $\alpha$ in
 \eqref{rla},
to the both sides of the model \eqref{eq1}
we reach
\begin{equation}
	\frac{\partial}{\partial t}u(\mbox{\boldmath${x}$}, t) + \varkappa D_{0t}^{-\alpha}\partial_{0t}^{\beta}u(\mbox{\boldmath${x}$}, t) + D_{0t}^{-\alpha}\mathcal{A}u(\mbox{\boldmath${x}$},t) = D_{0t}^{-\alpha}g(\mbox{\boldmath${x}$}, t) + \psi(\mbox{\boldmath${x}$})
,
	\label{eq4_1}
\end{equation}
in which 
$(\mbox{\boldmath${x}$}, t) \in \Omega\times(0,T].$
The operator $D_{0t}^{-\alpha}\partial_{0t}^{\beta}$,
 in  \eqref{eq4_1} is evaluated using the following calculation 
\begin{align*}
\nonumber
D_{0t}^{-\alpha}\partial_{0t}^{\beta} u(\mbox{\boldmath${x}$}, t) = 
&
 \int\limits_{0}^{t}\omega_{\alpha}^{(t-\eta)}
\int\limits_{0}^{\eta} \omega_{1-\beta}^{(\eta-\xi)}\frac{\partial u(\mbox{\boldmath${x}$}, \xi)}{\partial \xi} d\xi
\\ \nonumber 
= & 
\int\limits_{0}^{t}\frac{\partial u(\mbox{\boldmath${x}$}, \xi)}{\partial \xi} d\xi
\int\limits_{\xi}^{t} \omega_{\alpha}^{(t-\eta)}  \omega_{1-\beta}^{(\eta-\xi)} d\eta 
= 
\int\limits_{0}^{t}\omega_{1-\beta+\alpha}^{(t-\xi)}\frac{\partial u(\mbox{\boldmath${x}$}, \xi)}{\partial \xi} d\xi
.
\end{align*}
In this way, 
the operator $D_{0t}^{-\alpha}\partial_{0t}^{\beta}$
 can be categorized into the following three cases based on the sign of 
 $ \alpha - \beta$, 
\begin{align}
\label{Dab}
D_{0t}^{-\alpha}\partial_{0t}^{\beta}u(\mbox{\boldmath${x}$}, t) =
\begin{cases}
	D_{0t}^{-(\alpha-\beta)}u(\mbox{\boldmath${x}$}, t) -
\omega_{1-\beta+\alpha}^{(t)}
\phi(\mbox{\boldmath${x}$}),       &\text{if}\quad \alpha-\beta >0,\\
	u(\mbox{\boldmath${x}$}, t) - \phi(\mbox{\boldmath${x}$}),   &\text{if} \quad \alpha-\beta = 0,\\
	\partial_{0t}^{\beta-\alpha}u(\mbox{\boldmath${x}$}, t),      & \text{if} \quad \alpha-\beta < 0.
\end{cases}
\end{align}
By 
considering Eqs. \eqref{eq4_1} and \eqref{Dab}
and defining $\gamma=\beta-\alpha$, 
the time-fractional mixed SDDWE model
\eqref{eq1}
can be classified into the following  three cases
based on the sign of $\gamma$,
for
$(\mbox{\boldmath${x}$}, t) \in \Omega\times(0,T]$,
\begin{enumerate}
  \item\label{eq5_1}  $-1<\gamma<0$:
\begin{align}\label{c11}
\frac{\partial}{\partial t}u(\mbox{\boldmath${x}$}, t) + \varkappa D_{0t}^{\gamma}u(\mbox{\boldmath${x}$}, t) + D_{0t}^{-\alpha}\mathcal{A}u(\mbox{\boldmath${x}$},t) = f(\mbox{\boldmath${x}$},t),
\end{align}
  with
\begin{align*}
f(\mbox{\boldmath${x}$},t) = D_{0t}^{-\alpha}g(\mbox{\boldmath${x}$}, t) 
+ \frac{\varkappa \phi(\mbox{\boldmath${x}$}) t^{-\gamma}}{\Gamma(1-\gamma)} + \psi(\mbox{\boldmath${x}$}).
 \end{align*}
 \item \label{eq5_2} $\gamma = 0$: 
\begin{align}\label{c22}
		\frac{\partial}{\partial t}u(\mbox{\boldmath${x}$}, t) + \varkappa u(\mbox{\boldmath${x}$}, t) + D_{0t}^{-\alpha}\mathcal{A}u(\mbox{\boldmath${x}$},t) = f(\mbox{\boldmath${x}$},t), 
\end{align}
  with
\begin{align*}
f(\mbox{\boldmath${x}$},t) = D_{0t}^{-\alpha}g(\mbox{\boldmath${x}$}, t) + \varkappa\phi(\mbox{\boldmath${x}$}) +  \psi(\mbox{\boldmath${x}$}).
\end{align*}
    \item   \label{eq5_3} $0<\gamma<1$:
\begin{align}\label{c33}
		\frac{\partial}{\partial t}u(\mbox{\boldmath${x}$}, t) + \varkappa \partial_{0t}^{\gamma}u(\mbox{\boldmath${x}$}, t) + D_{0t}^{-\alpha}\mathcal{A}u(\mbox{\boldmath${x}$},t) = f(\mbox{\boldmath${x}$}, t),
\end{align}
with 
\begin{align*}
f(\mbox{\boldmath${x}$},t) = D_{0t}^{-\alpha}g(\mbox{\boldmath${x}$}, t) +  \psi(\mbox{\boldmath${x}$}).
\end{align*}
\end{enumerate}
In Eqs. \eqref{c11}--\eqref{c33}, as $t$ approaches to 0, 
we have
$u_t(\mbox{\boldmath${x}$}, 0) = \psi(\mbox{\boldmath${x}$})$. 
In this way, the second initial condition in \eqref{eq3} can be derived.

By combining equations \eqref{c11}--\eqref{c33}, 
the initial-boundary value problem \eqref{eq1}--\eqref{eq3}
 is transformed into the following general 
 initial-boundary value problem
\begin{align}\label{eq6_1}
&
	\frac{\partial}{\partial t}u(\mbox{\boldmath${x}$}, t) + \varkappa_1 \partial_{0t}^{\gamma}u(\mbox{\boldmath${x}$}, t) + \varkappa_2 u(\mbox{\boldmath${x}$}, t) + \varkappa_3 D_{0t}^{-\delta}u(\mbox{\boldmath${x}$}, t) + D_{0t}^{-\alpha}\mathcal{A}u(\mbox{\boldmath${x}$},t) = f(\mbox{\boldmath${x}$},t), 
\\&
	u(\mbox{\boldmath${x}$},t) = 0, \quad \mbox{\boldmath${x}$}\in \partial \Omega,
	\label{eq6_2}
\\&
	u(\mbox{\boldmath${x}$}, 0) = \phi(x), \quad \mbox{\boldmath${x}$}\in \overline \Omega
.
\label{eq6_3}
\end{align}
In which
\begin{align*}
(\mbox{\boldmath${x}$}, t) \in \Omega\times(0,T],
\quad
0<\gamma<1,
\quad
0<\delta<1,
 \quad
\varkappa_i\geq0,~ i=1,2,3. 
\end{align*}
The primary focus of this study is to design and analyze
 a second-order difference scheme for the the initial-boundary value model 
\eqref{eq6_1}--\eqref{eq6_3}.

\subsection{\bf A priori estimate for a differential problem} \label{subsec:4.1}

The uniqueness and the continuous dependence of the solution to the model  \eqref{eq6_1}--\eqref{eq6_3} on the input data
can be inferred  from the a priori estimate given below.
%
\begin{theorem}\label{theorem_estim}
	The solution $u(\mbox{\boldmath${x}$}, t) $ of the problem \eqref{eq6_1}--\eqref{eq6_3} satisfies the following  a   priori
estimate
\begin{align}
\|u(t)\|^2 + \varkappa_1 D_{0t}^{\gamma-1}\|u(t)\|^2 + 2\varkappa_2 \int\limits_{0}^{t}\|u(s)\|^2ds 
 \leq 
C_1 \left(\|u_0\|^2 + \int\limits_{0}^{t}\|f(s)\|^2ds\right), 
\label{eq7_7}  
\end{align}
where $C_1=\max\left\{2+\varkappa_1 T^{1-\gamma}(3-\gamma)/\Gamma(3-\gamma),\, 4T\right\}$.
\end{theorem}
%
In order to establish the validity of Theorem \eqref{theorem_estim}, some auxiliary lemmas and results are required
\begin{lemma}\cite{vabishchevich2022numerical} \label{lem_positive}
	Assume that
$\mathcal{K}(t)\in L_{1}(0,T)$ is real-valued function and for $t>0$ satisfies the following properties 
\begin{align*}
\mathcal{K}(t) \geq 0, \quad
\frac{d }{dt}\mathcal{K}(t)\leq 0, \quad
\frac{d^2 }{dt^2}\mathcal{K}(t) \geq 0.
\end{align*}
Then it follows that the following inequality holds true
	\begin{equation}
	\int\limits_{0}^{t} u(s)\int\limits_{0}^{s}\mathcal{K}(s-\eta)u{(\eta)}d\eta ds\geq 0, \quad v(t)\in\mathcal{C}[0,T].
    \label{eq7_2}  
    \end{equation}
\end{lemma}
%
\begin{lemma}\cite{alikhanov2010priori} \label{lem_eneq_C_D}
For any absolutely continuous function $u(t)$ on the interval $[0, T]$, the following inequality holds.
    $$
    u(t)\partial_{0t}^{\nu}u(t)\geq \frac{1}{2}\partial_{0t}^{\nu}u^2(t),
    \quad
    0<\nu<1.
    $$
\end{lemma}
%
\begin{corollary}\label{cor_1}
For any absolutely continuous function $u(t)$ on the interval $[0, T]$, the following inequalities  hold.
	\begin{align}\label{eq7_4}  
		\int\limits_{0}^{t}u(s)\partial_{0s}^{\gamma}u(s)ds \geq \frac{1}{2}D_{0t}^{\gamma-1}u^2(t) - 
\frac{1}{2} \omega_{2-\gamma}^{(t)} u^2(0), 
\quad
 0<\gamma<1
 ,
	\end{align}
	\begin{align}
		\int\limits_{0}^{t}u(s)D_{0s}^{-\delta}u(s)ds \geq 0, 
\quad
 0<\delta<1.
	\end{align}
\end{corollary}
%
\begin{proof}[Proof of Theorem \ref{theorem_estim}]
In the space $\mathcal{H}$,
taking inner-product of equation \eqref{eq6_1} with
$u(\mbox{\boldmath${x}$}, t)$,
then changing the variable $t$ to $s$ 
and integrating over the time variable from
 $0$ to $t$, we get
	\begin{align}
\nonumber
\int\limits_{0}^{t}(u, \frac{\partial}{\partial t}u)ds+&
\varkappa_1\int\limits_{0}^{t}(u,  \partial_{0 s}^{\gamma}u)ds + \varkappa_2\int\limits_{0}^{t}(u, u)ds+ \varkappa_3\int\limits_{0}^{t}(u, D_{0 s}^{-\delta}u)ds 
\\
+ &
\int\limits_{0}^{t}(u, D_{0 s}^{-\alpha}\mathcal{A}u)ds = \int\limits_{0}^{t}(u,f)ds.
	\label{eq7_8}
	\end{align}
Taking into account the following inequality:
\begin{align*}
(u(s),f(s))
\leq 
\varepsilon \|u(s)\|^2 + \frac{1}{4\varepsilon}  \|f(s)\|^2
, \quad
\varepsilon >0
,
\end{align*}
as well as corollary \ref{cor_1}, by omitting positive terms we reach
\begin{align}
\nonumber
		\|u(t)\|^2 + &
\varkappa_1 D_{0t}^{\gamma-1}\|u(t)\|^2 + 2\varkappa_2 \int\limits_{0}^{t}\|u(s)\|^2ds
 \\&
 \leq 
 2\varepsilon \int\limits_{0}^{t}\|u(s)\|^2ds + 
\left(1+\varkappa_1 \omega_{2-\gamma}^{(t)} \right)\|u_0\|^2 + \frac{1}{2\varepsilon}\int\limits_{0}^{t}\|f(s)\|^2ds.
	\label{eq7_9}  
\end{align}
If $\varkappa_2>0$, by choosing $\varepsilon=\varkappa_2/2$ one can
 reach the a priori estimate
(\ref{eq7_7}).
However, for general case $\varkappa_2 \geq 0$
 we need to estimate $\int\limits_{0}^{t}\|u(s)\|^2ds$.
For this purpose,
we integrate \eqref{eq7_9} with respect to the variable $t$ form $0$ to $t$.
Set  $\varepsilon=1/(4T)$,
taking into account the following  inequality: 
\begin{align*}
\int\limits_{0}^{t}d\xi\int\limits_{0}^{\xi}\|u(s)\|^2ds = \int\limits_{0}^{t}(t-s)\|u(s)\|^2ds\leq T\int\limits_{0}^{t}\|u(s)\|^2ds
,
\end{align*}
 we reach
\begin{equation}
		\int\limits_{0}^{t}\|u(s)\|^2ds 
\leq 2
\left(T+\varkappa_1 \omega_{3-\gamma}^{(T)} \right)\|u_0\|^2 + 4T^2\int\limits_{0}^{t}\|f(s)\|^2ds.
	\label{eq7_10}  
\end{equation}
Now, the validity of the a priori estimate (\ref{eq7_7}) follows from (\ref{eq7_9}) and (\ref{eq7_10}), for $\varepsilon=1/(4T)$.
\end{proof}

\subsection{\bf A second order difference scheme}\label{design}
The objective of this subsection is to introduce a second-order difference scheme for approximating the model \eqref{eq6_1}--\eqref{eq6_3}. 
 To accomplish this objective, it is necessary to devise discrete methods that accurately approximate the first derivative, 
the Caputo fractional derivative and the Riemann-Liouville fractional integral. 
\\
We first propose a discrete analog to approximate the Riemann-Liouville fractional integrals in \eqref{eq6_1}. 
Consider the follwing uniform temporal grid  
$\bar \omega_\tau =\{t_j = j\tau, \, j=0,1,\ldots,N,\, \tau=T/N\}$.
 The fractional Riemann-Liouville integral at the point $t=t_{j+1}$ is approximated using the following generalized trapezoidal formula
\begin{align}
\nonumber
D_{0t_{j+1}}^{-\nu}v(t) = & 
\sum\limits_{r=0}^{j}\int\limits_{t_{r}}^{t_{r+1}} \omega_\nu^{(t_{j+1}-\xi)} v(\xi)d\xi
\\ \nonumber
 \approx &
\sum\limits_{r=0}^{j}\int\limits_{t_{r}}^{t_{r+1}} \omega_\nu^{(t_{j+1}-\xi)} \left(v(t_{r+1})\frac{\xi-t_r}{\tau}+v(t_{r})\frac{t_{r+1}-\xi}{\tau}\right)v(\xi)d\xi
\\ \label{eq8_0}
= &
\frac{\tau^{\nu}}{\Gamma(\nu+2)}\left(\sum\limits_{r=0}^{j}c_{j-r}^{(\nu)}v^{r+1} + \bar c_{j+1}^{(\nu)}v^0 \right) = \Delta_{0t_{j+1}}^{-\nu}v,
\end{align}
where  
$
c_0^{(\nu)} = 1,$
and for $r\geq1$ we have 
\begin{align}
 c_r^{(\nu)}& = (r+1)^{\nu+1} -2r^{\nu+1}+(r-1)^{\nu+1},
 \\
 \bar c_r^{(\nu)}&=(\nu+1)(r+1)^\nu - \left((r+1)^{\nu+1}-r^{\nu+1}\right)
 .
\end{align}
%
The truncation error of the operator $\Delta_{0t_{j+1}}^{-\nu }$ in \eqref{eq8_0}
is characterized by the following lemma.
\begin{lemma}\label{errorRL} 
Let $v(t) \in C^2[0,t_{j+1}]$.
	For any $\nu \in (0,1)$,
	$\Delta_{0t_{j+1}}^{-\nu  }u$
	is defined in \eqref{eq8_0}. Then we have
	\begin{align}\label{L3-2error}
		\left| \mathcal{D}_{0t_{j+1}}^{-\nu }v(t) - \Delta_{0t_{j+1}}^{-\nu }v \right|=
		\mathcal{O} \left( \tau^{2} \right).
	\end{align}
\end{lemma}
\begin{proof}
	By using the Lagrange interpolation remainder formula, it follows that
\begin{align*}
\nonumber 
\left|  D_{0t_{j+1}}^{-\nu}v(t) - \Delta_{0t_{j+1}}^{-\nu}v\right| 
& \leq   \sum\limits_{r=0}^{j}\int\limits_{t_{r}}^{t_{r+1}}\omega_\nu^{(t_{j+1}-\xi)}\frac{M_2}{2}(t_{r+1}-\xi)(\xi-t_r)d\xi
\\&
\leq\frac{ M_2}{8} \tau^2 \sum\limits_{r=0}^{j}\int\limits_{t_{r}}^{t_{r+1}}\omega_\nu^{(t_{j+1}-\xi)}d\xi 
= \frac{ M_2}{8}\tau^2\int\limits_{0}^{t_{j+1}}\omega_\nu^{(t_{j+1}-\xi)} d\xi
\leq
\frac{M_2 \omega_{\nu+1}^{(T)}}{8}
\tau^2 = \mathcal{O}(\tau^2),
\end{align*}
Here
$ M_2 = \max\limits_{0\leq t\leq T}|v''(t)| $.
\end{proof}
Next, we proceed with approximating the Caputo fractional derivative
 in \eqref{eq6_1}. 
 We utilize the recently developed L2 formula presented in 
 \cite{alikhanov2021high}
   to approximate the derivative at the point $t=t_{j+1}$
($j=1, 2, \ldots, N-1$). This method is as follows:
\begin{equation}\label{eq8}
\partial_{0t_{j+1}}^{\nu} u(t) =
\Delta_{0t_{j+1}}^{\nu}u(t) +R^{j+1}= 
\frac{\omega_{2-\nu}^{(\tau)} }{\tau} \sum\limits_{r=0}^{j}a_{j-r}^{(\nu)}(u^{r+1}-u^r)
+R^{j+1},
\quad 0<\nu<1
,
\end{equation}
for $j=1$
$$
	a_{r}^{(\nu)}=
	\begin{cases}
		b_{0}^{(\nu)}+d_{0}^{(\nu)}+d_{1}^{(\nu)}, & r=0,\\
		b_{1}^{(\nu)}-d_{0}^{(\nu)}-d_{1}^{(\nu)}, & r=1,
	\end{cases}
$$
for $j=2$
$$
	a_{r}^{(\nu)}=
	\begin{cases}
		b_{0}^{(\nu)}+d_{0}^{(\nu)},                                   & r=0,\\
		b_{1}^{(\nu)}-d_{0}^{(\nu)}+d_{1}^{(\nu)}+d_{2}^{(\nu)}, & r=1,\\
		b_{2}^{(\nu)}-d_{1}^{(\nu)}-d_{2}^{(\nu)},                  & r=2,
	\end{cases}
$$
%
%
for $j\geq 3$,
$$
	a_{r}^{(\nu )}=
	\begin{cases}
		b_{0}^{(\nu)}+d_{0}^{(\nu)},                                         &r=0,\\
		b_{r}^{(\nu)}-d_{r-1}^{(\nu)}+d_{r}^{(\nu)},                      & 1\leq r\leq j-2,\\
		b_{j-1}^{(\nu)}-d_{j-2}^{(\nu)}+d_{j-1}^{(\nu)}+d_{j}^{(\nu)}, & r=j-1,\\
		b_{j}^{(\nu)}-d_{j-1}^{(\nu)}-d_{j}^{(\nu)},                      & r=j,
	\end{cases}
$$
with
\begin{align}
&b_{r}^{(\nu)}=(r+1)^{1-\nu}-r^{1-\nu}, ~ 
\\
&d_{r}^{(\nu)}=\frac{1}{2-\nu}\left[(r+1)^{2-\nu}-r^{2-\nu}\right]-\frac{1}{2}\left[(r+1)^{1-\nu}+r^{1-\nu}\right],
\quad r\geq 0.
\end{align}
If  $u(t)\in \mathcal{C}^3[0,T]$, the truncation error of the L2 formula \eqref{eq8} satisfies
$R^{j+1}=\mathcal{O}\left(\tau^{3-\nu}\right)$.
\\
In order to attain  the desired convergence order of the proposed scheme,
 we employ a second-order finite difference formula to approximate the first derivative term in equation \eqref{eq6_1}. This approximation is given by
\begin{align}\label{extra}
	\dfrac{d}{d t}u(\mbox{\boldmath${x}$},t_{j+1}) =\frac{3u(\mbox{\boldmath${x}$},t_{j+1})
-4u(\mbox{\boldmath${x}$},t_j)
+u(\mbox{\boldmath${x}$},t_{j-1})} {2\tau}
	+ \mathcal{O}(\tau^2)
.
\end{align}
The validity of this approximation can be readily confirmed through the application of Taylor's Theorem.

Let
$y(\mbox{\boldmath${x}$},t_j) =y^j\in \mathcal{C}^{3}_{t}\left(\bar{\omega}_{\tau}\right)$  be the
numerical solution of $u(\mbox{\boldmath${x}$},t_j) $ on the domain $\bar{\omega}_{\tau}$.
Applying the temporal discretizations \eqref{eq8_0}, \eqref{eq8} and
 \eqref{extra} 
the differential model problem \eqref{eq6_1}--\eqref{eq6_3},
 we propose the following difference  scheme with the order of accuracy $\mathcal{O}\left( \tau^2 + \varepsilon(h) \right)$
to approximate the model
\begin{align}
&
	\mathcal{D}_{\tau}y^{j+1} + \Delta_{0t_{j+1}}^{-\alpha}\mathcal{A}_{h}y = \varphi^{j+1}, \quad j =1, 2, \ldots, N-1,
	\label{eq8_1}
\\&
 y^0=\phi,\quad y^1 = \phi + \psi\tau, 
 \label{eq8_2}
\end{align}
in which 
\begin{align}\label{md}
\mathcal{D}_{\tau}y^{j+1} = 	\frac{3y^{j+1}-4y^j+y^{j-1}}{2\tau} +\varkappa_1
\Delta_{0t_{j+1}}^{\gamma}y+\varkappa_2 y^{j+1}
+\varkappa_3
\Delta_{0t_{j+1}}^{-\delta}y
.
\end{align}
In \eqref{eq8_2}
at the initial time step $t_1$, we estimate the value of $y^1$ by employing the second-order
 Taylor's approximation.
The operator 
$\mathcal{A}_{h}$ is  some discrete  analog  of the corresponding elliptic operator $\mathcal{A}$ with error $\varepsilon(h) = \mathcal{A}u - \mathcal{A}_{h }u$. 
Some common approaches of $\mathcal{A}_{h}$ are 
the  ﬁnite  element method, the finite difference operator, 
 the compact difference operator, and
 the space spectral scheme.   
In what follows, we will assume that the operator $\mathcal{A}_{h}$ is also self-adjoint and positive-definite in the finite-dimensional Hilbert space $\mathcal{H}_h = L_2({\bar{\omega}_h})$, where $\bar{\omega}_h$ is some partition of the domain $\bar{\omega}$.
Furthermore,  $\varphi^j$ in \eqref{eq8_1},
 is an approximation of $f(\mbox{\boldmath${x}$},t_j)$
 in such a way that $\varphi^j-f(\mbox{\boldmath${x}$},t_j)=\mathcal{O}{(\tau^2 + \varepsilon(h))}$.

\section{Theoretical analysis of the proposed difference scheme}\label{conv}

Consider the error $e = y - u$. Substituting $y = e + u$ into the equation (\ref{eq8_1}) and the initial conditions (\ref{eq8_2}), we obtain the problem for the error as follows
\begin{align}
&
	\mathcal{D}_{\tau}e^{j+1} + \Delta_{0t_{j+1}}^{-\alpha}\mathcal{A}_{h}e = R^{j+1}, \quad j =1,2,\ldots, M,
	\label{eq9_1}
\\ &
	e^0 = 0, \quad e^1 = \mu, 
	\label{eq9_2}
\end{align}
where
\begin{align*}
&
R^{j+1} = -\mathcal{D}_{\tau}u^{j+1} - \Delta_{0t_{j+1}}^{-\alpha}\mathcal{A}_{h}u + \varphi^{j+1} = \mathcal{O}{(\tau^2 + \varepsilon(h))},
\\
&
\mu = -u^1+\phi + \psi\tau = \mathcal{O}(\tau^2).
\end{align*}
\begin{theorem}\label{thm_e}
The difference scheme (\ref{eq9_1})-(\ref{eq9_2}) is unconditionally stable and the following a priori estimate is valid for its solution
	\begin{equation}
	\|e^{j+1}\|^2\leq C_2
\left(\sum\limits_{s=1}^{j}\|R^{s+1}\|^2\tau+\| \mu \|^2
+\tau^{\alpha+1} 
\| \mu \|_{\mathcal{A}_{h}}^2
\right),
	\label{eq9_3}
	\end{equation}
where 
$C_2>0$
 is a constant independent of $\tau$ and $h$.
\end{theorem}
To prove Theorem \ref{thm_e},
 it is crucial to present the subsequent Lemmas.
\begin{lemma} \cite{alikhanov2021high} \label{lm_ineq} For any real constants $\kappa_0, \kappa_1$ such that
	$\kappa_0\geq\max\{\kappa_1,-3\kappa_1\}$, and $\{v_j\}_{j=0}^{j=M}$ the following
	inequality holds for $ j=1, 2, \ldots, N-1$ 
	\begin{equation}
		v_{j+1}\left(\kappa_0v_{j+1}-(\kappa_0-\kappa_1)v_{j}-\kappa_1v_{j-1}\right)\geq
		E_{j+1}(\kappa_0, \kappa_1)-E_{j}(\kappa_0, \kappa_1), 
\label{url5}
	\end{equation}
	where for $ j=1, 2, \ldots, N$ we have
\begin{align*}
	E_{j}(\kappa_0, \kappa_1)=&
\left(\frac{1}{2}\sqrt{\frac{\kappa_0-\kappa_1}{2}}+\frac{1}{2}\sqrt{\frac{\kappa_0+3\kappa_1}{2}}\right)^2v_{j}^{2}	
\\&+
\left(\sqrt{\frac{\kappa_0-\kappa_1}{2}}v_{j}-\left(\frac{1}{2}\sqrt{\frac{\kappa_0-\kappa_1}{2}}+\frac{1}{2}\sqrt{\frac{\kappa_0+3\kappa_1}{2}}\right)v_{j-1}\right)^2
. 
\end{align*}
\end{lemma}
\begin{lemma}\label{lem16}  \cite{alikhanov2021high} For any function $v(t)$ defined on the grid
	$\overline \omega_{\tau}$ the following inequality holds.
	\begin{align}\label{url6}		
v_{j+1} \Delta_{0t_{j+1}}^{\gamma}v
\geq
\frac{\tau^{-\gamma}}{\Gamma(2-\gamma)}\left(\mathcal{E}_{j+1}^{(\gamma)}-\mathcal{E}_{j}^{(\gamma)}\right) - \frac{\tau^{-\gamma} \bar c_j }{2\Gamma(2-\gamma)}v_0^2,
	\end{align}
	where  $ j=1, 2, 3, \ldots, N$,  and
	$$
 \mathcal{E}_{j}^{(\gamma)} = E_{j}^{(\gamma)} + \frac{1}{2}\sum\limits_{s=0}^{j-1}\bar{a}_{j-1-s}^{(\gamma)}v_{s+1}^2
 ,
 \quad 
 E_j^{(\gamma)}=E_j({a}_{0}^{(\gamma)}-{a}_{2}^{(\gamma)}, {a}_{1}^{(\gamma)}-{a}_{2}^{(\gamma)}),
	$$
	$$
	\bar{a}_{0}^{(\gamma)}={a}_{2}^{(\gamma)},\quad
	\bar{a}_{1}^{(\gamma)}={a}_{2}^{(\gamma)},\quad
	\bar{a}_{s}^{(\gamma)}={a}_{s}^{(\gamma)},\quad s=2,3,\ldots,j
.
	$$
\end{lemma}
\begin{lemma}\label{ThL} \cite{Thome1993}
	Assume that the sequence $\left\{a_n\right\}_{n = 0}^{\infty}$ of real numbers is satisfy the following properties 
	$$
	a_n \geq 0,
 \quad 
 a_{n} - a_{n+1}\geq 0, \quad 
 a_{n}-2a_{n+1}+a_{n+2} \geq 0, \quad n = 0, 1, \ldots
	$$
	Then
	$$
	\sum\limits_{s=1}^{n}\sum\limits_{k=1}^{s}a_{s-k}\xi_s\xi_k\geq\frac{1}{2}a_0\sum\limits_{s=1}^{n}\xi_s^2, \quad \forall \left(\xi_1, \xi_2, \ldots, \xi_n\right) \in \mathbf{R}^{n}.
	$$
\end{lemma}
\begin{remark}
Since Eq. \eqref{eq9_2} implies that $e^0=0$,
 our proposed scheme analysis does not involve the coefficient 
 $\bar c_{j+1}^{(\nu)}$ in \eqref{eq8_0}. 
We only consider the coefficient $c_{j-k}^{(\nu)}$ in \eqref{eq8_0}, 
which can be easily verified to satisfy the conditions 
of Lemma \ref{ThL}.
\end{remark}
%
\begin{proof}[Proof of Theorem \ref{thm_e}]
In the space $\mathcal{H}_h$,
multiply the equation (\ref{eq9_1}) scalarly  by $e^{j+1}\tau$ and replace $j$ with $s$,
then 
summing  over $s$ from $1$ to $j$ we get the following equality
\begin{equation} 
	\sum\limits_{s=1}^{j}\left(e^{s+1}, \mathcal{D}_{\tau}e^{s+1}\right)\tau + 	\sum\limits_{s=1}^{j}\left(e^{s+1}, \Delta_{0t_{s+1}}^{-\alpha}\mathcal{A}_{h}e\right)\tau = \sum\limits_{s=1}^{j}\left(e^{s+1}, R^{s+1}\right)\tau.
	\label{eq9_4}
	\end{equation}
We estimate each term of the equality \eqref{eq9_4} separately,
considering the deffiniation of the operator 
$ 
\mathcal{D}_{\tau}y^{j+1}$
in \eqref{md} and 
utilising  Lemmas \ref{lm_ineq}--\ref{ThL}.
%
\begin{align}\label{first}	
\nonumber &
\sum\limits_{s=1}^{j}\left(e^{s+1},\frac{3e^{s+1}-4e^s+e^{s-1}}{2\tau}\right)\tau 
\geq 
\sum\limits_{s=1}^{j}\left(E_{s+1}  -  E_{s} \right) = E_{j+1}  -  E_{1} , 
\\
&
	E_{s}  = \frac{1}{4}\|e^s\|^2+\|e^s-\frac{1}{2}e^{s-1}\|^2,
\end{align}
%
\begin{align}
\sum\limits_{s=1}^{j}&
\left(e^{s+1}, \Delta_{0t_{s+1}}^{\gamma}e\right)\tau \geq 
\omega_{2-\gamma}^{(\tau)}
\sum\limits_{s=1}^{j}\left(\mathcal{E}_{s+1}^{(\gamma)}-\mathcal{E}_{s}^{(\gamma)}\right) = 
\omega_{2-\gamma}^{(\tau)}
\left(\mathcal{E}_{j+1}^{(\gamma)}-\mathcal{E}_{1}^{(\gamma)}\right),
\end{align}
In which, $\mathcal{E}_{j}^{(\gamma)}$ is defined in Lemma \ref{lem16}.
%
\begin{align}
\nonumber
\sum\limits_{s=1}^{j}
\left(e^{s+1}, \Delta_{0t_{s+1}}^{-\delta}e\right)\tau &=
\omega_{\delta+2}^{(\tau)}
\sum\limits_{s=1}^{j}\sum\limits_{k=0}^{s}c_{s-k}^{(\delta)}\left(e^{s+1}, e^{k+1}\right)
\\ \nonumber &
=
\omega_{\delta+2}^{(\tau)}
\left(
\sum\limits_{s=0}^{j}\sum\limits_{k=0}^{s}c_{s-k}^{(\delta)}\left(e^{s+1}, e^{k+1}\right) 
- 
c_{0}^{(\delta)}\|e^1\|^2
\right)
\\ 
& 
	\geq \frac{\tau^{\delta}}{2\Gamma(\delta+2)}\left(2c_{0}^{(\delta)}-2c_{1}^{(\delta)}+c_{2}^{(\delta)}\right)\sum\limits_{s=0}^{j}\|e^{s+1}\|^2\tau -\omega_{\delta+2}^{(\tau)}
c_{0}^{(\delta)}\|e^1\|^2,
\end{align}
\begin{align}
\nonumber 
\sum\limits_{s=1}^{j}
\left(e^{s+1}, \Delta_{0t_{s+1}}^{-\alpha}\mathcal{A}_{h}e\right)\tau &=
\omega_{\alpha+2}^{(\tau)}
\sum\limits_{s=1}^{j}\sum\limits_{k=0}^{s}c_{s-k}^{(\alpha)}\left(e^{s+1}, \mathcal{A}_{h}e^{k+1}\right)
\\ \nonumber
& =
\omega_{\alpha+2}^{(\tau)}
\sum\limits_{s=0}^{j}\sum\limits_{k=0}^{s}c_{s-k}^{(\alpha)}\left(e^{s+1}, \mathcal{A}_{h}e^{k+1}\right) 
- \omega_{\alpha+2}^{(\tau)}
c_{0}^{(\alpha )}\|e^1\|_{\mathcal{A}_{h}}^2
\\&
	\geq
 \frac{\tau^{\alpha}}{2\Gamma(\alpha+2)}
\left(2c_{0}^{(\alpha)}-2c_{1}^{(\alpha)}+c_{2}^{(\alpha)}\right)
\sum\limits_{s=0}^{j}\|e^{s+1}\|_{\mathcal{A}_{h}}^2\tau - 
\omega_{\alpha+2}^{(\tau)}
c_{0}^{(\alpha)}\|e^1\|_{\mathcal{A}_{h}}^2
.
\end{align}
Here, in the space $\mathcal{H}_h$,
we introduced the following  scaler poduct 
and the corresponding norm
\begin{align}\label{1nrm1}
(u,\mathcal{A}_{h}v)=(u,v)_{\mathcal{A}_{h}}, \quad
\|e\|_{\mathcal{A}_{h}}^2=\left(e, \mathcal{A}_{h}e\right)
.
\end{align}
For the right-hand side of \eqref{eq9_4}, we have 
\begin{align}\label{last}
\sum\limits_{s=1}^{j}\left(e^{s+1}, R^{s+1}\right)\tau
\leq
\varepsilon \sum\limits_{s=1}^{j}  \|e^{s+1}\|^2 \tau
+ \frac{1}{4\varepsilon} \sum\limits_{s=1}^{s+1} \| R^{s+1} \|^2 \tau.
\end{align}
By following a similar procedure to the proof of Theorem \ref{theorem_estim},
if
  $\varkappa_2>0$ or $\varkappa_2\geq 0$,
choosing  $\varepsilon=\varkappa_2$  or  $\varepsilon=1/(4T)$,
   respectively,
 and substituting Eqs. \eqref{first}--\eqref{last} into Eq. \eqref{eq9_4}, 
 while omitting positive terms
  completes the proof. 
\end{proof}

\begin{remark}
  (The multi-term model).

Consider a generalization of the single-therm time-fractional order mixed SDDWE model 
to the following multi-term time-fractional order problem
\begin{align}
\label{eq1M}
&
\partial_{0t}^{\alpha_0+1}\lambda_0 u(\mbox{\boldmath${x}$}, t) + 
\sum\limits_{r=1}^m \lambda_r \partial_{0t}^{\alpha_r+1}u(\mbox{\boldmath${x}$}, t)+ 
\sum\limits_{r=0}^m \omega_r \partial_{0t}^{\beta_r}u(\mbox{\boldmath${x}$}, t) + \mathcal{A}u(\mbox{\boldmath${x}$},t) = g(\mbox{\boldmath${x}$}, t),
\end{align}
with
 $0 < \alpha_{m} < \alpha_{m-1}< \cdots < \alpha_{0}<1$,
 $0< \beta_{m} < \beta_{m-1}< \cdots < \beta_{0}<1$,
 $\lambda_{r},\omega_{r} \geq 0$,
 $(\lambda_{0}>0 )$.
 To deal with this case, one needs to 
 apply the Riemann-Liouville fractional integration operator of order $\alpha_0$, 
 on both sides of the \eqref{eq1M}.
In this way, following the same procedure and arguments as in the single-term case,
we reach
\begin{align}
\nonumber
\frac{\partial}{\partial t} u(\mbox{\boldmath${x}$}, t) 
 &
 + \varkappa_1
\sum\limits_{r=1}^{m_1} \bar{\lambda}_r \partial_{0t}^{\gamma_r}u(\mbox{\boldmath${x}$}, t)
+ \varkappa_2 u(\mbox{\boldmath${x}$}, t)
+\varkappa_3
\sum\limits_{r=0}^{m_2} \bar{\omega}_r D_{0t}^{-\delta_r}u(\mbox{\boldmath${x}$}, t) 
\\&+ D_{0t}^{-\alpha_0}\mathcal{A}u(\mbox{\boldmath${x}$},t) =f(\mbox{\boldmath${x}$}, t)
.\label{eq2M}
\end{align} 
In which $m_1 \geq m$ and $m_2 \leq m$.
Designing a second-order temporal difference scheme for the  multi-term time-fractional order
problem \eqref{eq2M} and  analyzing 
convergence of the resulting difference scheme, 
can be accomplished using the same procedures described in 
subsection \ref{design} and  section \ref{conv}, respectively.
\end{remark}

\section{Numerical simulation }\label{num-sec}

In this section, numerical simulations for the model 
\eqref{eq6_1}--\eqref{eq6_3} are demonstrated to assess the reliability and efficiency of the suggested difference scheme
 \eqref{eq8_1}--\eqref{md}.
In the first example, the one-dimensional case is initially examined. In the second example, our proposed difference scheme   is employed to solve the multi-term time-fractional order model.
 Lastly, a two-dimensional model is explored in the third example. 
 Denote the errors
\begin{align}\label{eco}
E=\max\limits_{0\leq j\leq N}\|e^j\|, \quad
E_c=\max\limits_{0\leq j\leq N}\|e^j\|_{{C(\bar{\omega}_{h\tau})}}
, \quad
E_{\mathcal{A}_{h}}=\max\limits_{0\leq j\leq N} \|e^j\|_{\mathcal{A}_{h}} 
,
\end{align}
with the corresponding   convergence orders 
CO,
CO$_c$ and CO$_{\mathcal{A}_{h}}$ respectively, 
where
$\|\cdot\|$ is the L$_2$ norm,
$\|\cdot\|_{C(\bar{\omega}_{h\tau})}$
is the maximum norm defined by
$\|e\|_{C(\bar{\omega}_{h\tau})}=
\max\limits_{(x_i,t_j)\in \bar{\omega}_{h\tau} }|e|$,
and $\|\cdot\|_{\mathcal{A}_{h}}$ is
defined in \eqref{1nrm1}.
The temporal and spatial convergence orders are determined using the formulas
\begin{align*}
\text{CO}=
\begin{cases}
  \text{log}_{{h_1}/{h_2}} \frac{\|E(h_1,\tau)\|}{\|E(h_2,\tau)\|} , & \text{in space}, \\
  \text{log}_{{\tau_1}/{\tau_2}} \frac{\|E(h,\tau_1)\|}{\|E(h,\tau_2)\|} , & \text{in time}.
\end{cases}
\end{align*}
\begin{example}\label{ex_1}
One-dimensional single term model.
\\
Consider  the model 
 \eqref{eq6_1}--\eqref{eq6_3}   with 
  an exact analytical solution of the following form
\begin{align}
\label{exactex1}
u(x,t)=(t+ t^{2+\gamma}+t^{3+\delta}) \sin(2\pi x).
\end{align}
The initial and boundary conditions are determined based on the exact solution.
 For this problem, we choose
the elliptic operator $\mathcal{A}$  as 
\begin{align}
\label{L}
&
\mathcal{A}u=-\frac{\partial }{\partial x}\left(\varphi(x)\frac{\partial
		u}{\partial x}\right)+\varpi(x)u
,
\end{align}
with the coefficients 
$\varphi(x)=3-\cos(2x),$
$\varpi(x)=2-\sin(3x)$.
In the spatial direction, we use Lemma 3.1 from \cite{alikhanov2021high}
to discretize the operator $\mathcal{A}$   as 
\begin{align}\label{llambda}
\mathcal{A} u=\mathcal{A}_{h} u +\mathcal{O}(h^2)
,
\end{align}
where, for 
$i=1,\ldots,M-1,$
the difference operator  $\mathcal{A}_{h}$ is deﬁned as 
\begin{align}\label{ah-disc}
\mathcal{A}_{h} u_i=
-\frac{\varphi(x_{i+0.5})u_{i+1}-\left(\varphi(x_{i+0.5})+\varphi(x_{i-0.5})\right)u_i+\varphi(x_{i-0.5})u_{i-1}}{h^2}
+
\varpi(x_{i})u_i.
\end{align}
We discretize the spatial domain
$\Omega_{h}  $ with step size
$h= 1/M $.
In this way, the  uniform spatial grid is
$\Omega_{h} =\{x_{i}:~ x_{i} = i_1 h, i_1=0,1,\ldots,M\}$.
Since the convergence order of the spatial discretization 
 in equation \eqref{llambda}--\eqref{ah-disc}
is well-understood  in the literature (e.g., 
\cite{alikhanov2015new,alikhanov2021high,aslan2023}), our focus in   is on the temporal convergence and accuracies.
By utilizing our difference scheme, 
we solve this problem and report the numerical order of convergence in the temporal direction for various values of the derivative and integral orders $\gamma$, $\delta$ and $\alpha$.
The results are presented in Table  \ref{t1_ex1},
where different values of $\tau$ are considered 
while maintaining  fixed values of 
$h=1/4000$,  
$\varkappa_1=2$,
$\varkappa_2=4$,
$\varkappa_3=6$,
 and  $T=1$. 
 The Table demonstrates that the numerical scheme \eqref{eq8_1}--\eqref{md} 
 exhibits second-order accuracy in the temporal direction
 for solving the  model  \eqref{eq6_1}--\eqref{eq6_3},
 thereby corroborating our theoretical findings on convergence.
\end{example}
\begin{example}\label{ex_2}
One-dimensional multi-term time-fractional
order model.
\\
Consider  the model 
 \eqref{eq2M}   with 
 $m_1=m_2=3$, and
  an exact analytical solution of the following form
\begin{align}
\label{exactex1}
u(x,t)=(t+ t^{2}+t^{3}) \sin(4\pi x).
\end{align}
The initial and boundary conditions are determined based on the exact solution.
The elliptic operator $\mathcal{A}$ is chosen as described in equation \eqref{L},
and the coefficients of $\mathcal{A}$ are selected as
$\varphi(x)=2-\cos(x)$,
$\varpi(x)=1-\sin(x)$.
We discretize the spatial direction following the same approach used in example \eqref{ex_1}.
Employing our difference scheme, we solve the problem and evaluate the numerical order of convergence in the temporal direction for different values of 
$\gamma_r$, $\delta_r$ and $\alpha_0$. 
The results are presented in Table  \ref{t1_ex2}, with varying values of 
$\tau$ while maintaining  fixed values of 
$h = 1/7000$
$\varkappa_1=6$,
$\varkappa_2=2$,
$\varkappa_3=4$,
$\bar{\lambda}=(3,5,7)$,
$\bar{\omega}=(2,4,8)$,
and 
$T = 1$. 
The Table clearly demonstrates that the numerical scheme \eqref{eq8_1}--\eqref{md} 
  achieves second-order accuracy in the temporal direction when solving the multi-term time-fractional
order model \eqref{eq2M}. 
These findings provide further evidence that aligns with our theoretical analysis on convergence.
\end{example}
\begin{example}\label{ex_3}
Two-dimensional model.
\\
In order to
check the numerical accuracy and verify the convergence rates of the presented scheme, 
we consider  the following two-dimensional
problem for
$(\mbox{\boldmath${x}$},\mbox{\boldmath${y}$}) \in \Omega_x \times \Omega_y$,
 and $t \in (0, T]$,
\begin{align}
\nonumber
&
\frac{\partial}{\partial t}
u(\mbox{\boldmath${x}$},\mbox{\boldmath${y}$}, t) 
+ \varkappa_1 \partial_{0t}^{\gamma}
u(\mbox{\boldmath${x}$},\mbox{\boldmath${y}$}, t) 
+ \varkappa_2 u(\mbox{\boldmath${x}$},\mbox{\boldmath${y}$}, t) 
+ \varkappa_3 D_{0t}^{-\delta}
u(\mbox{\boldmath${x}$},\mbox{\boldmath${y}$}, t) 
\\ \label{2deq6_1}
& \qquad \qquad \quad
+ 
D_{0t}^{-\alpha}\mathcal{A}
u(\mbox{\boldmath${x}$},\mbox{\boldmath${y}$},t) = 
f(\mbox{\boldmath${x}$},\mbox{\boldmath${y}$},t), 
\\&
	u(\mbox{\boldmath${x}$},\mbox{\boldmath${y}$},t) = 0, \quad \mbox{\boldmath${x}$}\in \partial \Omega_x,
\\&
	u(\mbox{\boldmath${x}$},\mbox{\boldmath${y}$},t) = 0, \quad \mbox{\boldmath${y}$}\in \partial \Omega_y,
\\&
	u(\mbox{\boldmath${x}$}, 0) = \phi(x), \quad \mbox{\boldmath${x}$}\in \overline \Omega
.
\label{2deq6_3}
\end{align}
with
\begin{align*}
\mathcal{A}u= - 
\frac{\partial }{\partial x}\left(\phi(x,y)\frac{\partial
		u}{\partial x}\right)
-
\frac{\partial }{\partial y}\left(\phi(x,y)\frac{\partial
		u}{\partial y}\right)
+
\psi(x,y)u.
\end{align*}
We consider  this model 
  with 
  an exact analytical solution of the following form
\begin{align}
\label{exactex1}
u(x,t)=t^{2+\alpha_0} \sin(\pi x)
,
\end{align}
with the coefficients 
$\varphi(x)=1$,
$\varpi(x)=0$.
To utilize the numerical scheme \eqref{eq8_1}--\eqref{md},
 to approximate this model, 
we discretize the spatial domain
$\Omega_{h_1,h_2} \in \mathbb{R}^2 $ with step sizes
$h_1= 1/M_1 $ and $h_2= 1/M_2 $.
In this way, the new uniform spatial grid is
$\Omega_{h_1,h_2} =\{(x_{i_1},y_{i_2} ): x_{i_1} = i_1 h_1,\, y_{i_2} = i_2 h_2,\,  i_1=0,1,\ldots,M_1,\, i_2=0,1,\ldots,M_2  \}$.

Table \ref{t1_ex3} presents the errors and their corresponding temporal convergence order for different values of 
$\gamma$, $\delta$ and $\alpha$, with 
$h = 1/1000$
$\varkappa_1=1$,
$\varkappa_2=3$,
$\varkappa_3=5$,
and 
$T = 1$.
 The observed temporal convergence order of two aligns with the theoretical convergence results, even in the case of the two-dimensional model. 
 Additionally, Table \ref{t2_ex3} demonstrates the numerical order of convergence in the spatial direction for various spatial step sizes
 $h=h_1=h_2$, while keeping the time step size fixed at $1/600$ and
 $\varkappa_1=5$,
$\varkappa_2=3$,
$\varkappa_3=1$,
$T = 1$.
  The results from the table indicate that the spatial discretization formula \eqref{ah-disc} achieves the expected second-order accuracy in the spatial direction for the two-dimensional model.

  It is important to acknowledge that, regarding two-dimensional 
  time-fractional mixed SDDWE model,
 this study exclusively presents numerical results. However, the theoretical analysis of the proposed method's performance when applied to solve two-dimensional models 
  can be perform by the same procedure of the one-dimensional model.

\end{example}

\section{Conclusions} \label{con-sec}
We have developed a temporal second-order numerical scheme for solving time-fractional mixed SDDWEs and conducted a thorough analysis of the proposed method. 
The design of the method involved utilizing the Riemann-Liouville fractional integral operator to transform the problem into a fractional differential and integral equation with  orders belonging to $(0,1)$.
 Subsequently, a second-order analog was constructed to approximate the Riemann-Liouville fractional integral,
 and an L2-type formula of order $3-\alpha$, with $0<\alpha<1$ was employed to approximate the Caputo derivative. 
 The unconditional stability of the method in $L_2$ norms has been proven. 
 The presented method demonstrates its potential in constructing and analyzing a second-order L2-type numerical scheme for the mixed SDDWEs with multi-term time-fractional derivatives.
We conducted numerical experiments on three different examples, in which
 a one-dimensional model, a multi-term model and a two--dimensional model are analysed numerically. 
The numerical results confirmed the accuracy of the new approach and provided further support for our analytical findings.





\section*{Acknowledgements}
The research of 
Anatoly A. Alikhanov  
was supported by  Russian Science Foundation 
(Grant No. 22-21-00363).
The research of Chengming Huang was supported by National Natural Science Foundation of China (Grant No.  12171177).

%




\begin{table}[h!]
	\begin{center}
		\caption{
Error and temporal convergence order of example \ref{ex_1} with varying temporal step size $\tau=T/N$ 
for different parameters $\gamma$, $\delta$,  $\alpha $, 
while  fixed values of 
$\varkappa_1=2 $,
$\varkappa_2=4 $,
$\varkappa_3=6 $,
and
 $h=1/4000$, at $T=1$. }
		\begin{tabular}{llccccccccc} 
			\hline
 &$ N $ & $E $ &  CO  & $E_c$ & CO$_c$  & $E_{\mathcal{A}_{h}}$ & CO$_{\mathcal{A}_{h}}$
\\
\hline
$\gamma=0.9$ 
& 20 &  5.9025e-04 &              &  8.8843e-04  &              & 6.0193e-03  \\     
$\delta=0.5$
& 40 &  1.6894e-04 &     1.8048  &  2.5242e-04  &     1.8154   &  1.7212e-03  &     1.8061  \\     
$\alpha=0.1 $
& 80 &  4.6971e-05 &     1.8467  &  6.9817e-05  &     1.8542   &  4.7814e-04  &     1.8479  \\     
& 160 &  1.2632e-05 &     1.8946  &  1.8718e-05  &     1.8992   &  1.2851e-04  &     1.8955  \\     
& 320 &  3.1565e-06 &     2.0008  &  4.6796e-06  &     2.0000   &  3.2080e-05  &     2.0022  \\  
\hline
$\gamma=0.1$
& 20 &  1.5857e-03 &              &  2.3081e-03  &              & 1.6099e-02  \\     
$\delta=0.9$
& 40 &  4.0536e-04 &     1.9679  &  5.8994e-04  &     1.9681   &  4.1154e-03  &     1.9679  \\     
$\alpha=0.5 $
& 80 &  1.0267e-04 &     1.9811  &  1.4942e-04  &     1.9812   &  1.0423e-03  &     1.9813  \\     
& 160 &  2.5681e-05 &     1.9993  &  3.7382e-05  &     1.9990  &  2.6069e-04  &     1.9993  \\     
& 320 &  6.2007e-06 &     2.0502  &  9.0376e-06  &     2.0483  &  6.2933e-05  &     2.0504  \\ 
			\hline
$\gamma=0.5$
& 20 &  1.3755e-03 &              &  1.9881e-03  &             & 1.3948e-02  \\     
$\delta=0.1$
& 40 &  3.4526e-04 &     1.9943  &  4.9838e-04  &     1.9961   &  3.4999e-03  &     1.9947  \\     
$\alpha=0.9 $
& 80 &  8.6350e-05 &     1.9994  &  1.2455e-04  &     2.0006   &  8.7517e-04  &     1.9997  \\     
& 160 &  2.1402e-05 &     2.0124 &  3.0848e-05  &     2.0134   &  2.1691e-04  &     2.0125  \\     
& 320 &  5.1218e-06 &     2.0630 &  7.3761e-06  &     2.0642   &  5.1871e-05  &     2.0641  \\  
\hline
		\end{tabular}\label{t1_ex1}
	\end{center}
\end{table}

\begin{table}[h!]
	\begin{center}
		\caption{ 
Error and temporal convergence order of example \ref{ex_2} with varying temporal step size $\tau=T/N$ 
for different parameters $\gamma_r=(\gamma_1,\gamma_2,\gamma_3)$,
 $\delta_r=(\delta_1,\delta_2,\delta_3)$,  $\alpha_0 $, 
while  fixed values of 
$\lambda=(\bar{\lambda}_1,\bar{\lambda}_2,\bar{\lambda}_3)$,
$\omega=(\bar{\omega}_1,\bar{\omega}_2,\bar{\omega}_3)$,
$\varkappa_1=6$,
$\varkappa_2=2$,
$\varkappa_3=4$,
and
 $h=1/7000$, at $T=1$. }
		\begin{tabular}{llccccccccc} 
			\hline
 &$ N $ & $E $ &  CO  & $E_c$ & CO$_c$  & $E_{\mathcal{A}_{h}}$ & CO$_{\mathcal{A}_{h}}$
\\
\hline
$\gamma=(0.5,0.3,0.1)$ 
& 20 &  7.0665e-04 &              &  1.0626e-03  &               & 9.6255e-03  \\     
$\delta=(0.9,0.5,0.1)$ 
& 40 &  1.8162e-04 &     1.9601  &  2.7221e-04  &     1.9648   &  2.4732e-03  &     1.9605  \\     
$\alpha_0=0.6$
& 80 &  4.6164e-05 &     1.9761  &  6.9025e-05  &     1.9796   &  6.2847e-04  &     1.9764  \\     
$\omega=(2,4,8)$
& 160 &  1.1543e-05 &     1.9997  &  1.7220e-05  &     2.0030   &  1.5712e-04  &     2.0000  \\     
$\lambda=(3,5,7)$
& 320 &  2.7479e-06 &     2.0706  &  4.0830e-06  &     2.0764   &  3.7383e-05  &     2.0714  \\ 
\hline
$\gamma=(0.4,0.3,0.2)$
& 20 &  7.0665e-04 &              &  1.0626e-03  &               & 9.6255e-03  \\     
$\delta=(0.9,0.8,0.7)$
& 40 &  1.8162e-04 &     1.9601  &  2.7221e-04  &     1.9648   &  2.4732e-03  &     1.9605  \\     
$\alpha_0=0.5 $
& 80 &  4.6164e-05 &     1.9761  &  6.9025e-05  &     1.9796   &  6.2847e-04  &     1.9764  \\     
$\lambda=(3,5,7)$
& 160 &  1.1543e-05 &     1.9997  &  1.7220e-05  &     2.0030   &  1.5712e-04  &     2.0000  \\     
$\omega=(2,4,8)$
& 320 &  2.7479e-06 &     2.0706  &  4.0830e-06  &     2.0764   &  3.7383e-05  &     2.0714  \\ 
\hline
$\gamma=(0.8,0.7,0.6)$ 
& 20 &  3.6326e-04 &              &  6.4096e-04  &               & 5.0473e-03  \\     
$\delta=(0.3,0.2,0.1)$ 
& 40 &  1.0889e-04 &     1.7382  &  1.8422e-04  &     1.7988   &  1.5037e-03  &     1.7470  \\     
$\alpha_0=0.9$
& 80 &  3.0640e-05 &     1.8293  &  5.0459e-05  &     1.8683   &  4.2165e-04  &     1.8344  \\     
$\lambda=(3,5,7)$
& 160 &  8.2569e-06 &     1.8918  &  1.3344e-05  &     1.9189   &  1.1337e-04  &     1.8951  \\     
$\omega=(2,4,8)$
& 320 &  2.1011e-06 &     1.9745  &  3.3467e-06  &     1.9954   &  2.8795e-05  &     1.9771  \\   
\hline
		\end{tabular}\label{t1_ex2}
	\end{center}
\end{table}

\begin{table}[h!]
	\begin{center}
		\caption{
Error and temporal convergence order of example \ref{ex_3} with varying temporal step size $\tau=T/N$ 
for different parameters $\gamma$, $\delta$,  $\alpha $, 
while  fixed values of 
$\varkappa_1=1 $,
$\varkappa_2=3 $,
$\varkappa_3=5 $,
and
 $h=1/1000$, at $T=1$.
}
		\begin{tabular}{llccccccccc} 
			\hline
 &$ N $ & $E $ &  CO  & $E_c$ & CO$_c$  & $E_{\mathcal{A}_{h}}$ & CO$_{\mathcal{A}_{h}}$
\\
\hline
$\gamma=0.9$ 
& 10 &  4.4024e-04 &              &  8.8048e-04  &            & 1.9559e-03\\
$\delta=0.5$ 
& 20 &  1.1770e-04 &     1.9031  &  2.3541e-04  &     1.9031  &  5.2294e-04  &     1.9031\\
$\alpha=0.1 $
&40 &  3.1073e-05 &     1.9214  &  6.2146e-05  &     1.9214   &  1.3805e-04  &     1.9214\\
& 80 &  7.9946e-06 &     1.9586  &  1.5989e-05  &     1.9586   &  3.5519e-05  &     1.9586\\
& 160 &  1.8990e-06 &     2.0738  &  3.7981e-06  &     2.0738   &  8.4372e-06  &     2.0738\\
\hline
$\gamma=0.1$ 
& 10 &  1.3004e-03 &              &  2.6008e-03  &             & 5.7775e-03\\
$\delta=0.9$ 
& 20 &  3.2739e-04 &     1.9898  &  6.5479e-04  &     1.9898   &  1.4546e-03  &     1.9898\\
$\alpha=0.5 $
& 40 &  8.2798e-05 &     1.9834  &  1.6560e-04  &     1.9834   &  3.6786e-04  &     1.9834\\
& 80 &  2.0730e-05 &     1.9979  &  4.1460e-05  &     1.9979   &  9.2101e-05  &     1.9979\\
& 160 &  5.0397e-06 &     2.0403  &  1.0079e-05  &     2.0403   &  2.2391e-05  &     2.0403\\

\hline
$\gamma=0.5$
& 10 &  1.4239e-03 &             &  2.8478e-03  &               & 6.3262e-03\\
$\delta=0.1$
& 20 &  3.7347e-04 &     1.9308  &  7.4693e-04  &     1.9308   &  1.6593e-03  &     1.9308\\
$\alpha=0.9 $
&40 &  9.5757e-05 &     1.9635  &  1.9151e-04  &     1.9635   &  4.2544e-04  &     1.9635\\
& 80 &  2.4281e-05 &     1.9795  &  4.8562e-05  &     1.9795   &  1.0788e-04  &     1.9795\\
& 160 &  6.0565e-06 &     2.0033  &  1.2113e-05  &     2.0033   &  2.6908e-05  &     2.0033\\
\hline
		\end{tabular}\label{t1_ex3}
	\end{center}
\end{table}

\begin{table}[h!]
	\begin{center}
		\caption{
Error and spatial convergence order of example \ref{ex_3} with varying spatial step size 
 $h_1=h_2=1/M$,
for different parameters $\gamma$, $\delta$,  $\alpha $, 
while  fixed values of 
$\varkappa_1=5 $,
$\varkappa_2=3 $,
$\varkappa_3=1 $,
and
 $\tau=1/600$, at $T=1$.
}
		\begin{tabular}{llccccccccc} 
			\hline
 &$ M $ & $E $ &  CO  & $E_c$ & CO$_c$  & $E_{\mathcal{A}_{h}}$ & CO$_{\mathcal{A}_{h}}$
\\
\hline
$\gamma=0.9$ 
& 10 &  2.0708e-03 &              &  4.1416e-03  &              & 9.1625e-03\\
$\delta=0.5$
& 20 &  5.1741e-04 &     2.0008  &  1.0348e-03  &     2.0008   &  2.2964e-03  &     1.9964\\
$\alpha=0.1 $
& 40 &  1.2925e-04 &     2.0012  &  2.5850e-04  &     2.0012   &  5.7408e-04  &     2.0000\\
& 80 &  3.2220e-05 &     2.0041  &  6.4439e-05  &     2.0041   &  1.4314e-04  &     2.0039\\
& 160 &  7.9632e-06 &     2.0165 &  1.5926e-05  &     2.0165   &  3.5379e-05  &     2.0164\\
\hline
$\gamma=0.1$ 
& 10 &  2.0384e-03 &             &  4.0768e-03  &              & 9.0192e-03\\
$\delta=0.9$ 
& 20 &  5.0917e-04 &     2.0012  &  1.0183e-03  &     2.0012   &  2.2598e-03  &     1.9968\\
$\alpha=0.5$
& 40 &  1.2705e-04 &     2.0027  &  2.5410e-04  &     2.0027   &  5.6432e-04  &     2.0016\\
& 80 &  3.1533e-05 &     2.0104  &  6.3067e-05  &     2.0104   &  1.4009e-04  &     2.0102\\
& 160 &  7.6549e-06 &     2.0424 &  1.5310e-05  &     2.0424   &  3.4009e-05  &     2.0424\\
\hline
$\gamma=0.5$
& 10 &  1.0515e-03 &             &  2.1030e-03  &              & 4.6525e-03\\
$\delta=0.1$ 
& 20 &  2.6300e-04 &     1.9993  &  5.2600e-04  &     1.9993   &  1.1673e-03  &     1.9949\\
$\alpha=0.9 $
& 40 &  6.5545e-05 &     2.0045  &  1.3109e-04  &     2.0045   &  2.9113e-04  &     2.0034\\
& 80 &  1.6161e-05 &     2.0200  &  3.2321e-05  &     2.0200   &  7.1795e-05  &     2.0197\\
& 160 &  3.8131e-06 &     2.0834 &  7.6262e-06  &     2.0834  &  1.6941e-05  &     2.0834 \\
\hline
		\end{tabular}\label{t2_ex3}
	\end{center}
\end{table}

\end{document}